\documentclass[a4paper,12pt]{article}
\usepackage{amsmath,amsthm,amssymb}
\usepackage{times}
\usepackage{enumerate}
\usepackage[square,sort,comma,numbers]{natbib}
\usepackage{latexsym}
\usepackage{color,graphicx}
\usepackage{float}
\usepackage{enumitem}
\usepackage{scalerel,stackengine}

\RequirePackage[colorlinks,citecolor=blue,linkcolor=blue,urlcolor=blue]{hyperref}

\makeatletter
\def\namedlabel#1#2{\begingroup
	#2%
	\def\@currentlabel{#2}%
	\phantomsection\label{#1}\endgroup
}
\makeatother

\def\titlerunning#1{\gdef\titrun{#1}}
\makeatletter
\def\author#1{\gdef\autrun{\def\and{\unskip, }#1}\gdef\@author{#1}}

\makeatother

\def\subjclass#1{{\renewcommand{\thefootnote}{}%
		\footnote{\emph{Mathematics Subject Classification (2010):} #1}}}
\def\keywords#1{\par\medskip
	\noindent\textbf{Keywords.} #1}

\newtheorem{theorem}{Theorem}[section]

\newtheorem{lemma}[theorem]{Lemma}
\newtheorem{proposition}[theorem]{Proposition}

\theoremstyle{definition}
\newtheorem{definition}[theorem]{Definition}

\DeclareMathOperator{\cov}{Cov}
\DeclareMathOperator{\Var}{Var}

\numberwithin{equation}{section}

\newcommand{\eps}{\varepsilon}
\newcommand{\R}{\mathbb{R}}
\newcommand{\I}{\mathbb{I}}

\newcommand{\Z}{\mathbb{Z}}
\newcommand{\N}{\mathbb{N}}
\newcommand{\p}{\mathbb{P}}
\newcommand{\Cf}{\mathrm{C}}
\newcommand{\B}{\mathcal{B}}

\newcommand{\F}{\mathcal{F}}

\newcommand{\E}[1]{\ensuremath{\mathbb E \left[ #1 \right]}}

\newcommand{\e}{\ensuremath{\mathbb E}}

\newcommand{\cP}{\mathcal{P}}

\newcommand{\fM}{\mathfrak{M}}

\usepackage[colorinlistoftodos,prependcaption,color=yellow,textsize=tiny]{todonotes}

\title{A Central Limit Theorem for \\ Modified Massive Arratia Flow}
\titlerunning{CLT for MMAF}

\author{Andrey Dorogovtsev\footnote{Institute of Mathematics of NAS of Ukraine, Tereschenkivska st. 3, 01024 Kyiv, Ukraine. E-mail: \href{mailto:andrey.dorogovtsev@gmail.com}{andrey.dorogovtsev@gmail.com}} \and 
Vitalii Konarovskyi\footnote{Faculty of mathematics, informatics and natural sciences, University of Hamburg, Bundesstraße 55, 20146 Hamburg, Germany. E-mail: \href{mailto:vitalii.konarovskyi@uni-hamburg.de}{vitalii.konarovskyi@uni-hamburg.de}}\, $^*$  \and
Max von Renesse\footnote{Fakult\"{a}t f\"{u}r Mathematik und Informatik, Universit\"{a}t Leipzig, Augustusplatz 10, 04109 Leipzig, Germany. E-mail: \href{mailto:renesse@math.uni-leipzig.de}{renesse@math.uni-leipzig.de}}}

\def\subjclass#1{{\renewcommand{\thefootnote}{}%
\footnote{\emph{Mathematics Subject Classification (2020):} #1}}}

\def\keywords#1{{\renewcommand{\thefootnote}{}%
\footnote{\emph{Keywords:} #1}}}

\allowdisplaybreaks

\begin{document}

\maketitle

\begin{abstract}
The modified massive Arratia flow is a model for the dynamics of passive particle clusters moving in a random fluid that accounts for the effects of mass aggregation. We show a central limit theorem for the point process associated to the cluster positions when the system is started from a uniform configuration.  The critical mixing estimate is obtained by coupling the system to countably many independent Brownian motions.

\end{abstract}


\section{Introduction} 

The modified massive Arratia flow (MMAF) introduced in~\cite{Konarovskyi:TVP:2010,Konarovskyi:AP:2017}  is a model for the evolution of an ensemble of passive particle clusters in a random 1D-fluid. It describes the positions and sizes of a family of clusters moving as independent scaled Brownian motions on the real line as long as no collisions between clusters occur, the diffusion rate of each cluster being inverse proportional to its mass. In case of a collision clusters coalesce and form a new cluster of aggregate size (cf. Definition \ref{def:MMAF} below). This model is an overdamped intertial version of the well-studied Arratia flow model of coalescing Brownian motions that has no masses and no diffusivity rescaling \cite{Arratia:1979,Dorogovtsev:2004,Dorogovtsev:2010,Schertzer:2014,Berestycki:2015,Korenovska:2017,Schertzer:2017,Vovchanskii_Dor:2013,Glinyanaya:2018,Riabov:2017,Dorogovtsev_Vovchanskyi_2:2020,Dorogovtsev_book:2024}. Variants of the MMAF  for different initial conditions, asymptotic properties of the trajectories and its large deviations were investigated in \cite{Konarovskyi:TVP:2010,Konarovskyi:UMJ:2010,Konarovskyi:2014,Konarovskyi:EJP:2017,Konarovskyi:AP:2017,Marx:2018,Konarovskyi:CPAM:2019,Marx:2021,Konarovskyi:CD:2020}. A reversible extension of the MMAF model featuring a splitting mechanism was introduced and studied in \cite{Konarovskyi:CFWDPA:2022,Konarovskyi:CFWD:2024}.

In this note we study the point process induced from the cluster positions when the MMAF system is started from the uniform configuration on the integer lattice $\Z$. Unlike for the standard Arratia flow \cite{Munasinghe:2006,Tribe:2011} precise formulas or the structure of the associated point process of the cluster ensemble are yet unknown, but one can ask for asymptotic properties. The aim of this work is to show a central limit theorem for the  occupation measure in Theorem \ref{the_clt_for_mmaf} below. The statement is analogous to the corresponding result in \cite{Dorogovtsev_Hlyniana:2023} for the Arratia flow, however in the MMAF-case the correlation structure is more involved,  leading to a different approach for the critical mixing estimate in Section~\ref{sub:mixing_property_of_modified_massive_arratia_flow}. The main idea is based on coupling the MMAF to a countable family of independent Brownian motions, which is the main technical part of this work.

\section{Statement of Main Result}%

\label{sec:stationary_mmaf}

For a rigorous statement of our result we recall the definition of the MMAF from \cite{Konarovskyi:TVP:2010}, where we consider the special case of a uniform starting configuration on $\Z\subset \R$.

\begin{definition} \label{def:MMAF}

A family of continuous processes $\{ x_k(t),\ t\geq 0,\ k \in \Z \}$ is called a {\it modified massive Arratia flow started from $\Z$} if it satisfies the following properties 
\begin{enumerate}
  \item [(\namedlabel{F1}{F1})] for each $k \in \Z$ the process $x_k$ is a continuous square-integrable martingale with respect to the filtration 
    \[
      \F_t=\sigma(x_k(s),\ s\leq t,\ k \in \Z), \quad t\geq 0;
    \]

  \item [(\namedlabel{F2}{F2})] for each $k \in \Z$, $x_k(0)=k$;

  \item [(\namedlabel{F3}{F3})] for each $k<l$ from $\Z$ and $t\geq 0$, $x_k(t)\leq x_l(t)$;

  \item [(\namedlabel{F4}{F4})] for each $k \in \Z$ the quadratic variation of $x_k$ equals
    \[
      \langle x_k\rangle_t= \int_{ 0 }^{ t } \frac{ ds }{ m_k(s) }, \quad t\geq 0, 
    \]
    where $m_k(t)=\#\{ l:\ \exists s\leq t\ \ x_k(s)=x_l(s) \}$ and $\# A$ denotes a number of elements in $A$;
    
  \item [(\namedlabel{F5}{F5})] for each $k,l \in \Z$
    \[
      \langle x_k , x_l \rangle_{t\wedge\tau_{k,l}}=0, \quad t\geq 0,
    \]
    where $\tau_{k,l}=\inf\left\{ t:\ x_k(t)=x_l(t) \right\}$.
\end{enumerate}
\end{definition}

The existence of such a family of processes is shown in  \cite[Theorem 2]{Konarovskyi:TVP:2010}. Moreover, conditions~\eqref{F1}-\eqref{F5} uniquely determine the distributions of $(x_k)_{k \in \Z}$ in the space $\Cf\left( [0,\infty) \right)^{\Z}$.

We can introduce the associated \textit{occupation measure} induced on the real line by 

\[
  \mu_t(A)=\#(A \cap \{ x_k(t),\ k \in \Z \}), \quad A \in \B(\R ).
\]

Let  $\cP$ denote the set of bounded measurable functions $f: \R \to \R $ with period one, i.e. $f(x)=f(x+1)$ for all $x \in \R $.
For every $f \in \cP$ and $k \in \Z$ denote 
$$
  A_{k,t}f:= \int_{ k-1 }^{ k } f(u)\mu_t(du). 
$$
In Proposition~\ref{prop_finiteness_of_all_moments_for_bounded_functions} below, we show that the random variables $A_{k,t}f$, $k \in \Z$, have finite moments of every order. Moreover, it is easy to see that the sequence $(A_{k,t}f)_{k \in \Z}$ is stationary. This directly follows from the fact that the distributions of $(x_k)_{k \in \Z}$ and $(\tilde x_k^l)_{k \in \Z}$ coincides in $\Cf([0,\infty))^{\Z}$ for every $l \in \Z$, where $\tilde x_k^l(t)=x_{k+l}(t)-l$, $t\geq 0$.

With this our main result reads as follows.
\begin{theorem} %CLT for MMAF
  \label{the_clt_for_mmaf}
  For every $f \in \cP$ and $t>0$ the sequence
  \[
    Y_t^n(f)=\frac{ \sum_{ k=1 }^{ n }(A_{k,t}f-\E{ A_{k,t}f})  }{ \sqrt{ n } }, \quad n\geq 1,
  \]
  converges in distribution to a Gaussian random variable with mean 0 and variance 
  \[
    \sigma^2_t(f)=\Var{ A_{0,t}f } +2 \sum_{ k=1 }^{ \infty } \cov\left( A_{0,t}f,A_{k,t}f \right).
  \]
\end{theorem}

The proof of the statement is based on the  classical central limit theorem for stationary sequences, c.f\ \cite[Theorem~18.5.3]{Ibragimov:1971}. Hence, the main task is to establish quantitative mixing for the MMAF model which we achieve below by coupling to independent Brownian motions. 

\section{Mixing estimate }%
\label{sub:mixing_property_of_modified_massive_arratia_flow}

 Fixing $f \in \cP$ and $t>0$ and we introduce the mixing coefficient for $i \in \Z$ 
\[
  \alpha_i(j)= \sup\left\{ |\p(A \cap B)-\p(A)\p(B)|,\ A \in \fM_{-\infty}^i,\ B \in \fM_j^{\infty} \right\}, \quad j \in \Z,\ \ j>i,
\]
where $\fM_{a}^b=\sigma\{ A_{k,t}f,\ a\leq k\leq b \}$ for $-\infty\leq a<b\leq \infty$. Our goal is to prove the following proposition. 

\begin{proposition} %mixing property
  \label{pro_mixing_property}
  There exist constants $C>0$ and $\beta>0$ depending only on $f$ and $t$ such that
  \[
    \alpha_i(j)\leq C e^{-\beta \sqrt{ j-i }}
  \]
  for all $i<j$ from $\Z$.
\end{proposition}

The idea of proof of the proposition is to construct a modified massive Arratia flow in such a way that particles which came to $(-\infty,l]$ and $[k,+\infty)$ at time $t$ are ``almost'' independent. An important role is played by the following construction  from~\cite{Konarovskyi:TVP:2010}.

Let $T>0$ be fixed and $\Cf_{\leq }([0,T],\R^n)$ denote the Banach space of continuous functions $f:[0,T] \to \R^n$ satisfying $f_1(0)\leq \dots\leq f_n(0)$. We equip $\Cf_{\leq }([0,T],\R^n)$ with the uniform norm. We next define a map $F_n:\Cf_{\leq }([0,T], \R^{2n+1}) \to \Cf_{\leq }([0,T], \R^{2n+1} )$ as follows. For $f \in \Cf_{\leq}([0,T], \R ^{2n+1})$ we set $\tau^{(0)}=0$ and 
\[
  \Theta^{(0)}:=\{ \pi \subseteq  \{ -n,\dots,n \}:\ k,l \in \pi\iff f_k(0)=f_l(0) \}.
\]
Note that $\Theta^{(0)}$ is a partition of the set $\{ -n,\dots,n \}$. Let $\pi_k^{(0)}$ denote the set from $\Theta^{(0)}$ which contains $k \in \{ -n,\dots,n \}$. Define
\[
  f_k^{(0)}(t)=f_{\tilde{k}}(0)+ \frac{1}{ \sqrt{m^{(0)}_k} }\left(f_{\tilde{k}}^{(0)}(t)-f_{\tilde{k}}^{(0)}(0)\right), \quad t \in [0,T],
\]
where $\tilde{k}$ is the element from $\pi_k^{(0)}$ with the minimal absolute value and $m^{(0)}_k=\# \pi_k^{(0)}$. By induction, we construct $f^{(p)} \in \Cf_{\leq }([0,T],\R^{2n+1})$ for all $p \in \{ 1,\dots,2n \}$ as follows. For $k \in \{ -n,\dots,n-1 \}$, set 
\[
  \tau_k^{(p)}= \inf\left\{ t\geq 0:\ f_k^{(p-1)}(t)=f_{k+1}^{(p-1)}(t) \right\}\wedge T,
\]
and 
\[
  \tau^{(p)}= \inf\left\{ \tau_k^{(p)}>\tau^{(p-1)}:\ k \in \{ -n,\dots,n-1 \} \right\},
\]
where as usual $\inf \emptyset=+\infty$. Let $\Theta^{(p)}$ be the partition of $\{ -n,\dots,n \}$ defined by 
\[
  \Theta^{(p)}:=\{ \pi \subseteq \{ -n,\dots,n \}:\ k,l \in \pi \iff f_{k}^{(p-1)}(\tau^{(p)})=f_{l}^{(p-1)}(\tau^{(p)}) \}.
\]
Let $\pi_k^{(p)}$ be the element of $\Theta^{(p)}$ which contains $k \in \{ -n,\dots,n \}$. If $\tau^{(p)}=T$, we set $f_{k}^{(p)}(t)=f_k^{(p-1)}(t)$, $t \in [0,T]$. Otherwise, 
\[
  f_k^{(p)}(t)=
  \begin{cases}
    f_k^{(p-1)}(t), & \mbox{ if } 0\leq t< \tau^{(p)}, \\
    f_{\tilde{k}}^{(p-1)}(\tau^{(p)})+\frac{ \sqrt{ m_k^{(p-1)} } }{ \sqrt{ m_k^{(p)} } }\left( f_{\tilde{k}}^{(p-1)}(t)-f_{\tilde{k}}^{(p-1)}(\tau^{(p)}) \right), & \mbox{ if } \tau^{(p)}\leq t\leq T, 
  \end{cases}
\]
where $\tilde{k}$ is the element from $\pi_k^{(p)}$ with the minimal absolute value and $m^{(p)}_k=\# \pi_k^{(p)}$. We now define 
\[
  F_n(f)=f^{(2n)}.
\]

We also define maps
\[
  F_n^+:\Cf_{\leq }([0,T], \R^{n+1}) \to \Cf_{\leq }([0,T], \R^{n+1} )
\]
and 
\[
  F_n^-:\Cf_{\leq }([0,T], \R^{n+1}) \to \Cf_{\leq }([0,T], \R^{n+1} )
\]
similarly to $F_n$, where the set $\{ -n,\dots,n \}$ is replaced by $\{ 0,\dots,n \}$ and $\{ -n,\dots,0 \}$, respectively. By the construction, it is clear that maps $F_n$, $F_n^+$ and $F_n^-$ are measurable.

For $l \in \Z, n \in \N$ and a fixed family of independent Brownian motions $w_k(t)$, $t \in [0,T]$, $k \in \Z$, on $\R $ with diffusion rate 1 and $w_k(0)=k$, we define continuous stochastic processes $X^{l,n}_k$, $k \in \{ l-n,\dots,l+n \}$, $X^{l,n,+}_k$, $k \in \{ l,\dots,l+n \}$, and $X^{l,n,-}_k$, $k \in \{ l-n,\dots,l \}$, by
\[
  (X^{l,n}_{l-n},\dots,X^{l,n}_{l+n}):=F_n(w_{l-n},\dots,w_{l+n}), \quad (X^{l,n,+}_{l},\dots,X^{l,n,+}_{l+n}):=F_n^+(w_{l},\dots,w_{l+n})
\]
and 
\[
  (X^{l,n,-}_{n-l},\dots,X^{l,n,-}_{l}):=F_n^-(w_{n-l},\dots,w_l).
\]

We next recall the following lemma proved in~\cite[Lemma~5]{Konarovskyi:TVP:2010}.

\begin{lemma} %existence of geps
  \label{lem_existence_of_geps}
  Let $w_k(t)$, $t \in [0,T]$, $k \in \N_0=\N \cup \{ 0 \}$, be a family of independent Brownian motions on $\R $ with diffusion rate 1 and $w_k(0)=k$. Then for every $\eps \in \left(0, \frac{1}{ 2 }\right)$ the equality 
  \[
    \p\left\{ \max\limits_{ k \in \{ 0,\dots,n \} } \max\limits_{ t \in [0,T] }w_k(t)\leq n+ \frac{1}{ 2 }, \quad \min\limits_{ t \in [0,T] }w_{n+1}(t)> n+ \frac{1}{ 2 }+\eps\  \mbox{i.o.}  \right\}=1
  \]
  holds.
\end{lemma}

Let $\Z_{\leq l}=\{l,l-1,\dots\}$ and $\Z_{\geq l}=\{ l,l+1,\dots \}$. Using the lemma above and the construction of $F_n$, $F_n^+$ and $F_n^-$, similarly to the proof of \cite[Theorem~2]{Konarovskyi:TVP:2010} we get the following statement.

\begin{proposition} %special MMAF
  \label{pro_special_mmaf}
  Let $X^{l,n}_k$, $X^{l,n,+}_k$, $X^{l,n,-}_k$ be the stochastic processes constructed above. Then for every $l \in \Z$ the following statement holds.
  \begin{enumerate}
    \item [(i)] For every $k \in \Z$ the sequence $\{X^{l,n}_k,\ n\geq 1\}$ converges a.s. in the discrete topology of $\Cf[0,T]$ to a process $X^{l}_k$. Moreover, the family $X^l_k$, $k \in \Z$, is a modified massive Arratia flow started from $\Z$, i.e. it satisfies properties~\eqref{F1}-\eqref{F4}.

    \item [(ii)] For every $k \in \Z_{\geq l}$ the sequence $\{X^{l,n,+}_k,\ n\geq 1\}$ converges a.s. in the discrete topology of $\Cf[0,T]$ to a process $X^{l,+}_k$. Moreover, the family $X^{l,+}_k$, $k \in \Z_{\geq l}$,  satisfies properties~\eqref{F1}-\eqref{F4} with $\Z$ replaced by $\Z_{\geq l}$.

    \item [(iii)] For every $k \in \Z_{\leq  l}$ the sequence $\{X^{l,n,-}_k,\ n\geq 1\}$ converges a.s. in the discrete topology of $\Cf[0,T]$ to a process $X^{l,-}_k$. Moreover, the family $X^{l,-}_k$, $k \in \Z_{\leq l}$,  satisfies properties~\eqref{F1}-\eqref{F4} with $\Z$ replaced by $\Z_{\leq l}$.
  \end{enumerate}
\end{proposition}

We further define the events
\[
  A^+_{l,j}(t):=\left\{ \max\limits_{ k \in \{l,\dots,l+j\} } \max\limits_{ s \in [0,t] }w_{k}(s)\leq l+j+ \frac{1}{ 2 }, \quad \min\limits_{ s \in [0,t] }w_{l+j+1}(s)> l+j+ \frac{1}{ 2 } \right\}
\]
and 
\[
  A^-_{l,j}(t):=\left\{ \min\limits_{ k \in \{l-j,\dots,l\} } \min\limits_{ s \in [0,t] }w_{k}(s)\geq l-j- \frac{1}{ 2 }, \quad \max\limits_{ s \in [0,t] }w_{l-j-1}(s)< l-j- \frac{1}{ 2 } \right\}
\]
for all $l \in \Z$, $j \in \N_0$ and $t\in[0,T]$.

\begin{lemma} %gaps in MMAF
  \label{lem_gaps_in_mmaf}
  Let $l \in \Z$ and $j \in \N_0$ and let the families $\{X^{l}_k,\ k \in \Z\}$, $\{ X^{l,+}_k,\ k \in \Z_{\geq l} \}$, $\{ X^{l,-}_k,\ k \in \Z_{\leq l} \}$ be constructed above. Then $X^{l}_k=X^{l+p,+}_k$ on $A^{+}_{l,j}(T)$ for each $k> l+j$, $p \in \{ 0,\dots,j \}$, and $X^l_k=X^{l-p,-}_k$ on $A^{-}_{l,j}(T)$ for each $k<l-j$, $p \in \{ 0,\dots,j \}$.
\end{lemma}

\begin{proof} %
  The statement of the lemma directly follows from the construction of the families of random processes $X^l_k$, $X^{l,+}_k$, $X^{l,-}_k$ and the events $A^+_{l,j}(T)$, $A^-_{l,j}(T)$.
\end{proof}

We denote 
\[
  B^{+}_{l,N}(t)= \bigcup_{ j=1 }^{ N } A^+_{l,j}(t), \quad B^-_{l,N}(t)=\bigcup_{ j=1 }^{ N } A^{-}_{l,j}(t), \quad l \in \Z,\ \ N \in \N,\ \ t\in[0,T].
\]

\begin{lemma} %estimate of appearing gaps in intervals
  \label{lem_estimate_of_appearing_gaps_in_intervals}
  For each $T>0$ there exist a constant $C=C_T>0$ and a function $\beta_T(t):(0,T]\to(0,\infty)$ depending only on $T$ such that $t\beta_T(t)\to \frac{1}{8\sqrt{2}}$ as $t\to 0+$ and for every $l\in \Z$ and $N\in \N$  
  \[
    \p\left( B_{l,N}^{\pm}(t) \right)\geq 1-C e^{-\beta_T(t) \left[(\sqrt{ N }-\sqrt{2})\vee 1\right]}
  \]
  for all $l \in \Z$, $N \in \N$ and $t\in[0,T]$.
\end{lemma}
\begin{proof} %
  Note that $\p\left( B_{l,N}^+(t) \right)=\p\left( B_{l,N}^-(t) \right)=\p\left( B_{0,N}^+(t) \right)$. Therefore, it is enough to estimate $\p\left( B_{0,N}^+(t) \right)$ for each $N \in \N$. We denote 
  \[
    M_{j,n}(t)=\left\{ \max\limits_{ k \in \{ 0,\dots,j \} } \max\limits_{ s \in [0,t] }w_k(s)=\max\limits_{ k \in \{ j-n+2,\dots,j \} }\max\limits_{ s \in [0,t] }w_k(s) \right\}
  \]
  and 
  \[
    R_{j,n}(t)=\left\{ \max\limits_{ k \in \{ j-n+2,\dots,j \} }\max\limits_{ s \in [0,t] }w_k(s)\leq j+ \frac{1}{ 2 }, \min\limits_{ s \in [0,t] }w_{j+1}(s)> j+ \frac{1}{ 2 } \right\}
  \]
  for all $2\leq n\leq j$ and $t\in[0,T]$. Then, by Funibi's theorem,
  \begin{equation} %estimate of Mc
  \label{equ_estimate_of_mc}
   \begin{split}
    \p\left( M_{j,n}(t)^c \right)&\leq \p\left\{ \max\limits_{ k \in \{ 0,\dots,j-n+1 \} }\max\limits_{ s \in [0,t] }w_k(s)> j \right\}\le\sum_{ k=0 }^{ j-n+1 } \p\left\{ \max\limits_{ s \in [0,t] }w_k(s)> j \right\}\\ 
    &\le\sum_{ k=0 }^{ j-n+1 }\frac{ \sqrt{ 2 } }{ \sqrt{ \pi t} }\int_{ j-k }^{ +\infty } e^{- \frac{ u^2 }{ 2t }}du
    \le\sum_{ k=n-1 }^{ \infty }\frac{ \sqrt{ 2 } }{ \sqrt{ \pi t} }\int_{ k }^{ +\infty } e^{- \frac{ u^2 }{ 2t }}du\\
    &\leq \frac{ \sqrt{ 2 } }{ \sqrt{ \pi t} }  \int_{ n-2 }^{ +\infty } \left(\int_{ x }^{ +\infty } e^{- \frac{ u^2 }{ 2t }}du\right) dx \\
    &\leq \frac{ \sqrt{ 2 } }{ \sqrt{ \pi t} }  \int_{ n-2 }^{ +\infty } u e^{- \frac{ u^2 }{ 2t }}du
    \le \frac{\sqrt {2t}}{\sqrt{\pi}}e^{-\frac{ (n-2)^2 }{ 2t }}.
    \end{split}
  \end{equation}
 
  Using the independence of $w_k$, $k\in\Z$, we obtain
  \begin{equation} %estimate of R
  \label{equ_estimate_of_r}
    \begin{split}
      \p\left( R_{j,n}(t) \right)&=  \frac{\sqrt 2}{\sqrt{\pi t}}\int^{ \frac{1}{ 2 } }_{ 0 } e^{-\frac{ u^2 }{ 2t }}du \prod_{ k=j-n+2 }^{ j } \frac{ \sqrt{ 2 } }{ \sqrt{ \pi t } }\int^{ j+ \frac{1}{ 2 }-k }_{ 0 } e^{-\frac{ u^2 }{ 2t }}du  \\
      &= \left(1-\frac{\sqrt 2}{\sqrt{\pi t}}\int_{ \frac{1}{ 2 } }^{ \infty } e^{-\frac{ u^2 }{ 2t }}du  \right)\prod_{ k=0 }^{ n-2 } \left( 1-\frac{ \sqrt{ 2 } }{ \sqrt{ \pi t } } \int_{ k+ \frac{1}{ 2 } }^{ +\infty } e^{-\frac{ u^2 }{ 2t }}du  \right)
    \end{split}
  \end{equation}
  for all $2\leq n\leq j$ and $t\in[0,T]$. Taking a constant $a_T>0$ such that $\ln(1-x)\ge -a_Tx$ for all $0\le x\le \frac{\sqrt 2}{\sqrt{\pi T}}\int_{ \frac{1}{ 2 } }^{ \infty } e^{-\frac{ u^2 }{ 2T }}du=\frac{\sqrt 2}{\sqrt{\pi}}\int_{ \frac{1}{ 2T } }^{ \infty } e^{-\frac{ u^2 }{ 2 }}du$, we next estimate
  \begin{align*}
      \ln  \p\left( R_{j,n}(t) \right)&=\ln\left(1-\frac{\sqrt 2}{\sqrt{\pi t}}\int_{ \frac{1}{ 2 } }^{ \infty } e^{-\frac{ u^2 }{ 2t }}du  \right)+\sum_{ k=0 }^{ n-2 } \ln \left( 1-\frac{ \sqrt{ 2 } }{ \sqrt{ \pi t } } \int_{ k+ \frac{1}{ 2 } }^{ +\infty } e^{-\frac{ u^2 }{ 2t }}du  \right)\\
      &\ge-\frac{\sqrt 2a_T}{\sqrt{\pi t}}\int_{ \frac{1}{ 2 } }^{ \infty } e^{-\frac{ u^2 }{ 2t }}du-\sum_{ k=0 }^{ \infty } \frac{ \sqrt{ 2 }a_T }{ \sqrt{ \pi t } } \int_{ k+ \frac{1}{ 2 } }^{ +\infty } e^{-\frac{ u^2 }{ 2t }}du\\
      &\ge -\frac{2\sqrt 2a_T}{\sqrt{\pi t}}\int_{ \frac{1}{ 2 } }^{ \infty } e^{-\frac{ u^2 }{ 2t }}du-\frac{a_T \sqrt{ 2 } }{ \sqrt{ \pi t } } \int_{\frac{1}{2}}^\infty \int_{ x }^{ +\infty } e^{-\frac{ u^2 }{ 2t }}dudx\ge -\tilde{C}_T e^{-\frac{1}{8t}},
  \end{align*}
  where $\tilde C_T:=\frac{10\sqrt {2T}a_T}{\sqrt{\pi}}$.
  Thus, for each $2\le n\le j$ and $t\in[0,T]$, we get
  \[
    \p\left( R_{j,n}(t) \right)\ge e^{-\tilde C_T e^{-1/8t}}.
  \]

  Fix $p \in \Z_{\geq 3}$ and set $N_k=p+k$, $k \in \{ 1,\dots,p \}$,
  \[
    G_k(t)=R_{\tilde{N}_k,N_k}(t) \cap M_{\tilde{N}_k,N_k}(t),
  \]
  where $\tilde{N}_k=\sum_{ i=1 }^{ k } N_i$. Since $ \bigcup_{ k=1 }^{ p } G_k(t) \subseteq B_{0,N}^+(t)$ for $N=\tilde{N}_p=\frac{ (3p+1)p }{ 2 }$, one can estimate
  \begin{align*}
  \p\left( B_{0,N}^+(t) \right)&\geq \p\left( \bigcup_{ k=1 }^{ p } G_k(t) \right)= \p\left( \bigcup_{ k=1 }^{ p } \left(R_{\tilde{N}_k,N_k}(t) \setminus M_{\tilde{N}_k,N_k}(t)^c \right) \right)\\
  &\geq \p\left( \left(\bigcup_{ k=1 }^{ p } R_{\tilde{N}_k,N_k}(t)\right) \setminus\left(\bigcup_{ k=1 }^{ p }  M_{\tilde{N}_k,N_k}(t)^c \right) \right)\\
  &\geq \p\left( \bigcup_{ k=1 }^{ p } R_{\tilde{N}_k,N_k}(t)\right)-\p\left(\bigcup_{ k=1 }^{ p }  M_{\tilde{N}_k,N_k}(t)^c \right)\\
  &\geq  1-\p\left( \bigcap_{ k=1 }^{ p } R_{\tilde{N}_k,N_k}(t)^c\right)- C_T\sum_{ k=1 }^{ p } e^{- \frac{ (N_k-2)^2 }{ 2t }},
  \end{align*}
  where we used~\eqref{equ_estimate_of_mc} in the last step. Note that  
  \begin{equation} %estimate of sum
  \label{equ_estimate_of_sum}
  \sum_{ k=1 }^{ p } e^{-\frac{ (N_k-2)^2 }{ 2t }}=\sum_{ k=1 }^{ p } e^{-\frac{ (p+k-2)^2 }{ 2t }}\leq \int_{ p-2 }^{ +\infty } e^{-\frac{ u^2 }{ 2t }}du\leq 2Te^{-\frac{ (p-2)^2 }{ 2t }}, 
  \end{equation}
  since $p\geq 3$.  Set $\gamma_T:=-\ln\left(1-e^{-\tilde{C}_T e^{-1/8T}}\right)>0$. Using the inequality $1-e^{-x}\le x$ for $x\ge 0$, we get
  \[
  1-e^{-\tilde{C}_T e^{-1/8T}}\le \tilde{C}_Te^{-\frac{1}{8T}}=e^{\ln \tilde C_T -\frac{1}{8T}}
  \]
  Hence, by the independence of $R_{\tilde{N}_k,N_k}$, $k \in \{ 1,\dots,p \}$, and estimate~\eqref{equ_estimate_of_sum}, there exists a constant $C>0$ depending on $T$ such that
    \begin{align*}
    \p\left( B_{0,N}^+(t) \right)&\geq 1-\prod_{ k=1 }^{ p } \p\left( R_{\tilde{N_k},N_k}(t)^c \right)-C e^{-\frac{ (p-2)^2 }{ 2t }}.
  \end{align*}
  Using the  observation~\eqref{equ_estimate_of_r} and the inequality $1-e^{-x}\le x$ for $x\ge 0$, we continue  
  \begin{align*}
    \p\left( B_{0,N}^+(t) \right)&\geq 1-(1-e^{-\tilde{C}_T e^{-1/8t}})^p-Ce^{-\frac{ (p-2)^2 }{ 2t }}\\
    &\geq 1-e^{-\left[\left(\frac{1}{8t}-\ln \tilde C_T\right) \vee\gamma_T\right]p}-Ce^{-\frac{ (p-2)^2 }{ 2t }} \ge 1-C_Te^{-\beta_T(t)p}
  \end{align*}
  for all $p \in \Z_{\geq 3}$ and $N=\frac{ (3p+1)p }{ 2 }$, where  $\beta_T(t)=\left[\left(\frac{1}{8t}-\ln \tilde C_T\right)\vee\gamma_T\right]\wedge \frac{1}{t}$ and the constant $C_T$ depends only on $T$.
  Next, for every $N \in \Z_{\geq 18}$ and $p=\big\lfloor \sqrt{ N/2 }\big\rfloor $ we have $N\geq \frac{ (3p+1)p }{ 2 }$. Hence,  
  \begin{align*}
    \p\left( B_{0,N}^+(t) \right)&\geq \p\left( B_{0,\frac{ (3p+1)p }{ 2 }}^+(t) \right)\geq 1-C_Te^{-\beta_T(t) p}=1-C_Te^{-\beta_T(t) \left\lfloor \sqrt{ \frac{ N }{ 2 } } \right\rfloor}\\
    &\geq 1-C_T e^{-\beta_T(t) \left( \frac{ \sqrt{ N } }{ \sqrt{ 2 } }-1 \right)}
  \end{align*}
  for every $N\geq 18$.

  It only remains to get the estimate for $N<18$.  We first note that 
  \begin{align*}
  \p\left(A_{l,j}^+(t)^c\right)&\le(j+2)\p\left\{\max_{s\in[0,t]}w_0(s)\ge \frac{1}{2}\right\}\\
  &\le  \frac{\sqrt 2 (j+2)}{\sqrt{\pi t}}\int_{ \frac{1}{ 2 } }^{ \infty } e^{-\frac{ u^2 }{ 2t }}du\le \frac{4\sqrt {2T} (j+2)}{\sqrt{\pi}}e^{-\frac{ 1 }{ 8t }}.
  \end{align*}
  Thus, for all integer numbers $N<18$
  \begin{align*}
      \p\left( B_{0,N}^+(t)^c \right)\ge 1-\sum_{j=1}^N\frac{4\sqrt {2T} (j+2)}{\sqrt{\pi}}e^{-\frac{ 1 }{ 8t }}\ge 1- C_T e^{-\frac{1}{8t}}.
  \end{align*}
  This completes the proof of the lemma.
\end{proof}

\begin{proof}[Proof of Proposition~\ref{pro_mixing_property}] %
  Let $X_k^{l}$, $k \in \Z$, $X_k^{l,+}$, $k \in \Z_{\geq l}$ and $X_k^{l,-}$, $k \in \Z_{\leq l}$, be the families of processes constructed in Proposition~\ref{pro_special_mmaf} for each $l \in \Z$. We fix $i<j$ from $\Z$, $t>0$ and define for $l=\left\lfloor \frac{ i+j }{ 2 } \right\rfloor$ the measures
  \[
    \mu_t^l(A)=\#\left( A \cap \{ X_k^l(t),\ k \in \Z \} \right), \quad
  \]
  and 
  \[
    \mu_t^{l,+}(A)=\#\left( A \cap \{ X_k^{l,+},\ k \in \Z_{\geq l} \} \right), \quad \mu_t^{l,-}(A)=\#\left( A \cap \{ X_k^{l,-},\ k \in \Z_{\leq l} \} \right),
  \]
  for all $A \in \B(\R )$. Set also 
  \[
    A_{k,t}^lf= \int_{ k-1 }^{ k } f(u)\mu_t^l(du) \quad \mbox{and} \quad A_{k,t}^{\pm}f=\int_{ k-1 }^{ k } f(u)\mu_t^{l,\pm}(du), \quad k \in \Z. 
  \]
  By Proposition~\ref{pro_special_mmaf}, the distributions of $(A_{k,t}^lf)_{k \in \Z}$ and $(A_{k,t}f)_{k \in \Z}$ coincide in $\R^{\Z}$.

  We next assume that $j-i\geq 3$ and take arbitrary sets $A \in \fM_{-\infty}^i$, $B \in \fM_{j}^{+\infty}$. Then there exist Borel measurable sets $\tilde{A} \subseteq \R^{\Z_{\leq i}}$ and $\tilde{B} \subseteq \R^{\Z_{\geq j}}$ such that 
  \[
    A=\{ (A_{k,t}f)_{k \in \Z_{\leq i}} \in \tilde{A} \} \quad \mbox{and} \quad B=\{ (A_{k,t}f)_{k \in \Z_{\geq j}} \in \tilde{B} \}.
  \]
  Then, using Lemma~\ref{lem_estimate_of_appearing_gaps_in_intervals}, we can estimate 
  \begin{align*}
    &|\p\left( A \cap B \right)-\p(A)\p(B)|= \Big|\p\left( \left\{ (A_{k,t}f)_{k \in \Z_{\leq i}} \in \tilde{A} \right\} \cap \left\{ (A_{k,t}f)_{k \in \Z_{\geq j}} \in \tilde{B}\right\} \right)\\
    &\qquad\qquad-\p\left\{ (A_{k,t}f)_{k \in \Z_{\leq i}} \in \tilde{A} \right\}  \p\left\{ (A_{k,t}f)_{k \in \Z_{\geq j}} \in \tilde{B}\right\}\Big|\\
    &= \Big|\p\left( \left\{ (A_{k,t}^lf)_{k \in \Z_{\leq i}} \in \tilde{A} \right\} \cap \left\{ (A_{k,t}^lf)_{k \in \Z_{\geq j}} \in \tilde{B}\right\} \right)\\
    &\qquad\qquad-\p\left\{ (A_{k,t}^lf)_{k \in \Z_{\leq i}} \in \tilde{A} \right\}  \p\left\{ (A_{k,t}^lf)_{k \in \Z_{\geq j}} \in \tilde{B}\right\}\Big|\\
    &\leq  \Big|\p\left( \left\{ (A_{k,t}^lf)_{k \in \Z_{\leq i}} \in \tilde{A} \right\} \cap \left\{ (A_{k,t}^lf)_{k \in \Z_{\geq j}} \in \tilde{B}\right\} \cap B^+_{l,j-l} \cap B^-_{l,l-i} \right)\\
    &\qquad\qquad-\p\left(\left\{ (A_{k,t}^lf)_{k \in \Z_{\leq i}} \in \tilde{A} \right\} \cap B^-_{l,l-i}\right)  \p\left(\left\{ (A_{k,t}^lf)_{k \in \Z_{\geq j}} \in \tilde{B}\right\} \cap B_{l,j-l}^+\right)\Big|\\
    &+ Ce^{-\beta \sqrt{ j-i }}
  \end{align*}
  for some positive constants $C$, $\beta$ independent on $i,j$ and $A$, $B$.  By Lemmas~\ref{lem_gaps_in_mmaf} and~\ref{lem_estimate_of_appearing_gaps_in_intervals}, we get 
  \begin{align*}
    &|\p(A \cap B)-\p(A)\p(B)|\\
    &\leq \Big|\p\left( \left\{ (A_{k,t}^{l,-}f)_{k \in \Z_{\leq i}} \in \tilde{A} \right\} \cap \left\{ (A_{k,t}^{l+1,+}f)_{k \in \Z_{\geq j}} \in \tilde{B}\right\} \cap B^+_{l,j-l} \cap B^-_{l,l-i} \right)\\
    &\qquad\qquad-\p\left(\left\{ (A_{k,t}^{l,-}f)_{k \in \Z_{\leq i}} \in \tilde{A} \right\} \cap B^-_{l,l-i}\right)  \p\left(\left\{ (A_{k,t}^{l+1,+}f)_{k \in \Z_{\geq j}} \in \tilde{B}\right\} \cap B_{l,j-l}^+\right)\Big|\\
    &+ Ce^{-\beta \sqrt{ j-i }}\\
    &\leq \Big|\p\left( \left\{ (A_{k,t}^{l,-}f)_{k \in \Z_{\leq i}} \in \tilde{A} \right\} \cap \left\{ (A_{k,t}^{l+1,+}f)_{k \in \Z_{\geq j}} \in \tilde{B}\right\} \right)\\
    &\qquad\qquad-\p\left\{ (A_{k,t}^{l,-}f)_{k \in \Z_{\leq i}} \in \tilde{A} \right\}  \p\left\{ (A_{k,t}^{l+1,+}f)_{k \in \Z_{\geq j}} \in \tilde{B}\right\}\Big|\\
    &+ C_1e^{-\beta \sqrt{ j-i }},
  \end{align*}
  where $C_1>0$ is a constant independent on $i,j$ and $A$, $B$. Hence, using independence of $(A_{k,t}^{l,-})_{k \in \Z\leq i}$ and $(A_{k,l}^{l+1,+})_{k \in \Z\geq j}$, we can conclude that
  \[
    |\p(A \cap B)-\p(A)\p(B)|\leq C_1e^{-\beta \sqrt{ j-i }}.
  \]
  Now, taking the supremum over $A \in \fM_{-\infty}^i$ and $B \in \fM_j^{+\infty}$, we obtain the statement of the proposition.
\end{proof}

\section{Proof of Theorem~\ref{the_clt_for_mmaf}}%
\label{sub:proof_of_theorem_the_clt_for_mmaf}

In this section, we will prove Theorem~\ref{the_clt_for_mmaf}. According to \cite[Theorem~18.5.3]{Ibragimov:1971} and Proposition~\ref{pro_mixing_property}, the statement of Theorem~\ref{the_clt_for_mmaf} follows from the fact that $\E{(A_{k,t}f)^{2+\delta}}<\infty$ for some $\delta>0$. We will show that $A_{k,t}f$ have finite moments of all orders.

\begin{proposition} %finiteness of all moments for bounded functions
  \label{prop_finiteness_of_all_moments_for_bounded_functions}
  Let $f: \R \to \R $ be a bounded measurable function. Then for every $a<b$ from $\R $, $T>0$ and $p\geq 1$ there exists a constant $C>0$ such that
  \[
    \E{ \left|\int_{ a }^{ b } f(u)\mu_t(du)\right|^p}<C  
  \]
  for all $t\in[0,T]$.
\end{proposition}

\begin{proof} %
  Set $\|f\|_\infty:=\sup_{u\in\R}|f(u)|$. Using the definition of $\mu_t$, we estimate 
  \begin{align*}
    \E{ \left|\int_{ a }^{ b } f(u)\mu_t(du)\right|^p}
    &\leq \|f\|_\infty^p\E{ \mu_t([a,b])^{p} }\\
    &=\|f\|_\infty^p \E{ \left( \sum_{ l \in \Z } \I_{\left\{ x_{l}(t) \in [a,b] \right\}} \right)^p }\\
    &= \|f\|_\infty^p\sum_{ k<n }   \E{ \left( \sum_{ l=k }^n \I_{\left\{ x_{l}(t) \in [a,b] \right\}} \right)^p \I_{B_{k,n}^{a,b}(t)}} 
  \end{align*}
  where $B_{k,n}^{a,b}(t):=\left\{ x_{k-1}(t)<a,x_{k}(t)\geq a,x_n(t)\leq b,x_{n+1}(t)>b \right\}$. By the Cauchy-Schwartz inequality, we get 
  \begin{align*}
    \E{ \left|\int_{ a }^{ b } f(u)\mu_t(du)\right|^p}&\leq \|f\|_\infty^p \sum_{ k<n }    \E {\left( \sum_{ l=k }^{ n } \I_{\left\{ x_l(t)\in [a,b] \right\}} \right)^{2p}}^{ \frac{1}{ 2 }}\p(B_{k,n}^{a,b}(t))^{\frac{1}{2}}.
  \end{align*}
  Next, using H\"older's inequality, we obtain
  \begin{align*}
  \E{ \left|\int_{ a }^{ b } f(u)\mu_t(du)\right|^p}&\leq \|f\|_\infty^p\sum_{ k<n }   (n-k+1)^{p- \frac{1}{ 2 }}\E{ \sum_{ l=k }^{ n } \I_{\left\{ x_l(t)\in [a,b] \right\}}}^{ \frac{1}{ 2 }}\p(B_{k,n}^{a,b}(t))^{\frac 12}.
  \end{align*}
  Since $(x_{k+l}(t)-l)_{k \in \Z}$ and $(x_k(t))_{k \in \Z}$ have the same distributions, we conclude
  \begin{align*}
      \E{ \sum_{ l=k }^{ n } \I_{\left\{ x_l(t)\in [a,b] \right\}}}&=\sum_{ l=k }^{ n }\E{ \I_{\left\{ x_l(t)-l\in [-l+a,-l+b] \right\}} }\\
    &= \sum_{ l=k }^{ n }\E{ \I_{\left\{ x_0(t)\in [-l+a,-l+b] \right\}} }\le b-a+1.
  \end{align*}
  Consequently, we have
  \begin{align*}
  \E{ \left|\int_{ a }^{ b } f(u)\mu_t(du)\right|^p}&\leq \|f\|_\infty^p (b-a+1)^{ \frac{1}{ 2 }}\sum_{ k<n }   (n-k+1)^{p- \frac{1}{ 2 }}\p(B_{k,n}^{a,b}(t))^{\frac 12}.
  \end{align*}
  
  Now it remains to show that the series $\sum_{ k<n }   (n-k+1)^{p- \frac{1}{ 2 }}\p(B_{k,n}^{a,b}(t))^{\frac 12}$ converges and is uniformly bounded on $[0,T]$.
  We first estimate $\p\left\{ x_0(t)\geq c \right\}$ for every $c\geq 1$. According to \cite[Theorem~II.7.2']{Ikeda:1989}, there exists a Brownian motion $w(t)$, $t\geq 0$, probably on an extended probability space, such that $x_0(t)=w\left( \left\langle x_0 \right\rangle_t \right)$, $t\geq 0$. Since the quadratic variation 
  \[
    \left\langle x_0 \right\rangle_t= \int_{ 0 }^{ t } \frac{ds}{ m_0(s) }\leq t  
  \]
  for all $t\geq 0$, we get 
  \begin{equation} %estimate of x0 geq a
  \label{equ_estimate_of_x0_geq_a}
    \begin{split}
      \p\left\{ x_0(t)\geq c \right\}&\leq \p\left\{ \max\limits_{ s \in [0,t] }x_0(s)\geq c \right\} = \p\left\{ \max\limits_{ s \in [0,t] }w\left( \left\langle x_0 \right\rangle_s\right) \geq c \right\}\\
      &\leq \p\left\{ \max\limits_{ s \in [0,t] }w(s)\geq c \right\}= \frac{ \sqrt{ 2 } }{ \sqrt{ \pi t } }\int_{ c }^{ \infty } e^{- \frac{ u^2 }{ 2t }}du\leq \frac{ 2 \sqrt{ 2t } }{ \sqrt{ \pi } }e^{- \frac{ c^2 }{ 2t }}.  
    \end{split}
  \end{equation}
  Similarly, 
  \begin{equation} %estimate of x0 leq a
  \label{equ_estimate_of_x0_leq_a}
    \p\left\{ x_0(t)\leq -c \right\}\leq \frac{ 2\sqrt{ 2t } }{ \sqrt{ \pi } }e^{- \frac{ c^2 }{ 2t }}
  \end{equation}
  for all $c\geq 1$. We next rewrite
  \begin{equation} %esprasion for the sum
  \label{equ_esprasion_for_the_sum}
    \begin{split}
      \sum_{ k<n }   (n-k+1)^{p- \frac{1}{ 2 }}\p(B_{k,n}^{a,b}(t))^{\frac{1}{ 2 }}&= \sum_{ k=0 }^{ +\infty } \sum_{ n=k+1 }^{ +\infty } (n-k+1)^{p- \frac{1}{ 2 }}\p(B_{k,n}^{a,b}(t))^{\frac{1}{ 2 }}\\
      &+\sum_{ k=-\infty }^{ -1 } \sum_{ n=k+1 }^{ 0 } (n-k+1)^{p- \frac{1}{ 2 }}\p(B_{k,n}^{a,b}(t))^{\frac{1}{ 2 }}\\
      &+\sum_{ k=-\infty }^{-1} \sum_{ n=1 }^{ +\infty } (n-k+1)^{p- \frac{1}{ 2 }}\p(B_{k,n}^{a,b}(t))^{\frac{1}{ 2 }}.
    \end{split}
  \end{equation}
  In the first term of the right hand side of \eqref{equ_esprasion_for_the_sum}, we estimate 
  \begin{align*}
    \p\left( B_{k,n}^{a,b}(t) \right)&\leq \p\left\{ x_{k-1}(t)\leq a,x_n(t)\leq b \right\}=\E{ \I_{\left\{ x_{k-1}(t)\leq a \right\}}\I_{\left\{ x_n(t)\leq b \right\}} }\\
    &\leq \sqrt{ \E{ \I_{\left\{ x_{k-1}(t)\leq a \right\}}^2 } }\sqrt{ \E{ \I_{\left\{ x_n(t)\leq b \right\}}^2 } }\\
    &=\sqrt{ \p\left\{ x_{k-1}(t)\leq a \right\} }\sqrt{ \p\left\{ x_n(t)\leq b \right\} }.
  \end{align*}
  Since the distributions of the random variables $x_l(t)-l$ and $x_0(t)$ coincide, we get
  \begin{align*}
    \p\left\{ x_{k-1}(t)\leq a \right\}=\p\left\{ x_0(t)\leq -(k-1)+a \right\} \leq \frac{ 2 \sqrt{ 2t } }{ \sqrt{ \pi } }e^{- \frac{ (k-1-a)^2 }{ 2t }}
  \end{align*}
  for all $k\geq a+2$, and
  \begin{align*}
    \p\left\{ x_n(t)\leq b \right\}=\p\left\{ x_0(t)\leq -(n-b) \right\} \leq \frac{ 2 \sqrt{ 2t } }{ \sqrt{ \pi } }e^{-\frac{ (n-b)^2 }{ 2t }}
  \end{align*}
  for all $n\geq b+1$, by~\eqref{equ_estimate_of_x0_leq_a}. This implies that 
  $$
  \sum_{ k=0 }^{ +\infty } \sum_{ n=k+1 }^{ +\infty } (n-k+1)^{p- \frac{1}{ 2 }}\p(B_{k,n}^{a,b}(t))^{\frac 12}\le C
  $$ 
  for all $t\in[0,T]$ and some constant $C>0$.
  Similarly, estimating $\p( B_{k,n}^{a,b}(t) )$ by $\p\big\{ x_{k}(t)\geq a, x_{n+1}(t)\geq b \big\}$ and $\p\left\{ x_k(t)\geq a,x_n(t)\leq b \right\}$ in the second and third terms of the right hand side of~\eqref{equ_esprasion_for_the_sum}, respectively, and using~\eqref{equ_estimate_of_x0_geq_a},~\eqref{equ_estimate_of_x0_leq_a}, we get 
  \[
    \sum_{ k=-\infty }^{ -1 } \sum_{ n=k+1 }^{ 0 } (n-k+1)^{p- \frac{1}{ 2 }}\p(B_{k,n}^{a,b}(t))^{\frac 12}+\sum_{ k=-\infty }^{-1} \sum_{ n=1 }^{ +\infty } (n-k+1)^{p- \frac{1}{ 2 }}\p(B_{k,n}^{a,b}(t))^{\frac 12}\le C 
  \]
  for all $t\in[0,T]$.
This completes the proof of the lemma.
\end{proof}

We will next show that $\sigma^2_t(f)$ is strictly positive for some time $t>0$ and function $f$. The following lemma is true.
\begin{lemma}
    Let $f\in\Cf_b^3(\R)$ be an odd, 1-periodic function. Then $\frac{\sigma^2_t(f)}{t}\to \left(f'(0)\right)^2$ as $t\to 0+$. In particular, there exists $t>0$ such that $\sigma^2_t(f)>0$ if $f'(0)\not =0$. 
\end{lemma}

\begin{proof}
    For the proof of the lemma, we will use the fact that the particles
in the modified massive Arratia flow become more independent for small
times $t$.

Let 
\[
\tilde{A}_{k,t}f:=\int_{k-\frac{1}{2}}^{k+\frac{1}{2}}f(u)\mu_{t}(du)
\]
for each $k\in\Z$ and $t\ge0$. Using the fact that $f(-x)=-f(x)$ for all $x\in\R$ and the periodicity
of $f$, we conclude that $\e\big[\tilde{A}_{k,t}f\big]=0$. Let 
\[
\tilde{Y}_{t}^{n}(f):=\frac{1}{\sqrt{n}}\sum_{k=1}^{n}\tilde{A}_{k,t}f=\frac{1}{\sqrt{n}}\sum_{k=1}^{n}\left(\tilde{A}_{k,t}f-\E{\tilde{A}_{k,t}f}\right),\quad n\ge1.
\]
Similarly to the proof of Proposition~\ref{pro_mixing_property}, one can show that the family $\big(\tilde{A}_{k,t}f\big)_{k\in\Z}$ satisfies the
mixing condition with the same bound for the mixing coefficients.  Hence, $\big\{ \tilde{Y}_{t}^{n}(f)\big\} _{n\ge0}$
converges to a Gaussian random variable with mean $0$ and variance
\begin{align*}
\tilde{\sigma}_{t}^{2}(f) & :=\Var\tilde{A}_{0,t}f+2\sum_{k=1}^{\infty}\cov\left(\tilde{A}_{0,t}f,\tilde{A}_{k,t}f\right)\\
 & =\E{\left(\tilde{A}_{0,t}f\right)^{2}}+2\sum_{k=1}^{\infty}\E{\tilde{A}_{0,t}f\tilde{A}_{k,t}f}.
\end{align*}
On the other side, one can estimate 
\[
\E{\left(Y_{t}^{n}(f)-\tilde{Y}_{t}^{n}(f)\right)^{2}}=\frac{1}{n}\E{\left(\xi_{n}-\E{\xi_{n}}\right)^{2}}\le\frac{1}{n}\E{\xi_{n}^{2}},
\]
where 
\[
\xi_{n}:=\int_{0}^{\frac{1}{2}}f(u)\mu_{t}(du)-\int_{n}^{n+\frac{1}{2}}f(u)\mu_{t}(du).
\]
Due to Lemma~\ref{prop_finiteness_of_all_moments_for_bounded_functions} and the stationarity of the modified
massive Arratia flow, we estimate
\[
\E{\xi_{n}^{2}}\le2\E{\left(\int_{0}^{\frac{1}{2}}f(u)\mu_{t}(du)\right)^{2}}+2\E{\left(\int_{\frac{1}{2}}^{1}f(u)\mu_{t}(du)\right)^{2}}<\infty.
\]
Thus, $\e \big[\big(Y_{t}^{n}(f)-\tilde{Y}_{t}^{n}(f)\big)^{2}\big]\to0$
as $n\to\infty$. Therefore, according to Theorem~\ref{the_clt_for_mmaf}, $\big\{ \tilde{Y}_{t}^{n}(f)\big\} _{n\ge0}$
converges in distribution to the same limit as $\left\{ Y_{t}^{n}(f)\right\} _{n\ge0}$
. This implies that $\sigma_{t}^{2}(f)=\tilde{\sigma}_{t}^{2}(f)$.

Let $\left(w_{k}\right)_{k\in\Z}$ be a family of Brownian motions
that were used for the definition $X_{k}^{l,n}$ in Proposition~\ref{pro_special_mmaf}. Using Proposition~\ref{pro_special_mmaf}, we get
\begin{align}
\E{\left(\tilde{A}_{0,t}f\right)^{2}} & =\E{f^{2}(w_{0}(t))}\nonumber \\
 & +\E{\left(\left(\tilde{A}_{0,t}f\right)^{2}-f^{2}(w_{0}(t))\right)\I_{A}},\label{eq:expectation_of_tilde_A_0}
\end{align}
where $A:=\left\{ |X_{0}(t)|\ge\frac{1}{2}\right\} \cup\left\{ X_{-1}(t)\ge-\frac{1}{2}\right\} \cup\left\{ X_{1}(t)\le\frac{1}{2}\right\} .$
By Taylor's formula and the equality $f(0)=0$, that follows from the fact that $f$ is odd function,
the first term of the right hand side of the equality above can be
rewritten as
\begin{align*}
\E{f^{2}(w_{0}(t))} & =f^{2}(0)+\frac{1}{2}\frac{d^{2}f^{2}}{dx^{2}}(0)\E{w_{0}^{2}(t)}+o(t)\\
 & =\left(f'(0)\right)^{2}t+o(t).
\end{align*}
Using Hölder's inequality, the square of the second term of (\ref{eq:expectation_of_tilde_A_0})
can be estimated by
\[
I_{t}:=\left(\E{\left(\tilde{A}_{0,t}f\right)^{4}}+\E{f^{4}(w_{0}(t))}\right)\p\left(A\right).
\]
Now, by the boundedness of $f$ and Proposition~\ref{prop_finiteness_of_all_moments_for_bounded_functions}, there exists $C>0$ such that
\begin{align*}
I_{t} & \le C\left(\p\left\{ |X_{0}(t)|\ge\frac{1}{2}\right\} +\p\left\{ X_{-1}(t)\ge-\frac{1}{2}\right\} +\p\left\{ X_{1}(t)\le\frac{1}{2}\right\} \right)\\
 & \le3C\p\left\{ |X_{0}(t)|\ge\frac{1}{2}\right\} \le3C\p\left\{ |w_{0}(t)|\ge\frac{1}{2}\right\} \le Ce^{-\frac{1}{8t}}.
\end{align*}
Thus,
\[
\frac{1}{t}\E{\left(\tilde{A}_{0,t}f\right)^{2}}\to\left(f'(0)\right)^{2}
\]
as $t\to0$.

Similarly, we estimate $\e\big[\tilde{A}_{0,t}f\tilde{A}_{k,t}f\big]$ for
each $k\in\N$. Using the notation from the proof of Proposition~\ref{pro_mixing_property} and denoting
\[
\tilde{A}_{k,t}^{l}f:=\int_{k-\frac{1}{2}}^{k+\frac{1}{2}}f(u)\mu_{t}^{l}(du)\quad\mbox{and}\quad\tilde{A}_{k,t}^{\pm}f:=\int_{k-\frac{1}{2}}^{k+\frac{1}{2}}f(u)\mu_{t}^{l,\pm}(du),\quad k\in\Z,
\]
we estimate for $k\ge1$ and $l:=\left\lfloor \frac{k}{2}\right\rfloor $
\begin{align*}
\E{\tilde{A}_{0,t}f\tilde{A}_{k,t}f} & =\E{\tilde{A}_{0,t}^{l}f\tilde{A}_{k,t}^{l}f}\\
 & =\E{\tilde{A}_{0,t}^{l}f\tilde{A}_{k,t}^{l}f\I_{B_{l,k-l}^{+}\cap B_{l,l}^{-}}}+\E{\tilde{A}_{0,t}^{l}f\tilde{A}_{k,t}^{l}f\I_{\left(B_{l,k-l}^{+}\cap B_{l,l}^{-}\right)^{c}}}.
\end{align*}
By Lemma~\ref{lem_estimate_of_appearing_gaps_in_intervals},  Proposition~\ref{prop_finiteness_of_all_moments_for_bounded_functions}, and Hölder's inequality, the square of the
second term of the equality above can be estimated by
\[
\E{\left(\tilde{A}_{0,t}^{l}f\right)^{2}}\E{\left(\tilde{A}_{k,t}^{l}f\right)^{2}}\left(\p\left(\left(B_{l,k-l}^{+}\right)^{c}\right)+\p\left(\left(B_{l,l}^{-}\right)^{c}\right)\right)\le C_{T}e^{-\beta_{T}(t)\left[\left(\sqrt{k}-\sqrt{2}\right)\vee1\right]}
\]
for all $t\in[0,T]$, where the function $\beta_{T}:(0,T]\to(0,\infty)$
and the constant $C_{T}$ depend only on $T$ and $t\beta_{T}(t)\to\frac{1}{8\sqrt{2}}$
as $t\to0$. Using now Lemma~\ref{lem_gaps_in_mmaf}, we get
\begin{align*}
\E{\tilde{A}_{0,t}^{l}f\tilde{A}_{k,t}^{l}f\I_{B_{l,k-l}^{+}\cap B_{l,l}^{-}}} & =\E{\tilde{A}_{0,t}^{l,-}f\tilde{A}_{k,t}^{l+1,+}f\I_{B_{l,k-l}^{+}\cap B_{l,l}^{-}}}\\
 & =\E{\tilde{A}_{0,t}^{l,-}f\tilde{A}_{k,t}^{l+1,+}f}+\E{\tilde{A}_{0,t}^{l,-}f\tilde{A}_{k,t}^{l+1,+}f\I_{\left(B_{l,k-l}^{+}\cap B_{l,l}^{-}\right)^{c}}}.
\end{align*}
We can similarly estimate  the square of the second term of the equality
above by the expression $C_{T}e^{-\beta_{T}(t)\left[\left(\sqrt{k}-\sqrt{2}\right)\vee1\right]}$.
By the independence of $\tilde{A}_{0,t}^{l,-}f$ and $\tilde{A}_{k,t}^{l+1,+}f$,
the first term of the equality above equals $\e\big[\tilde{A}_{0,t}^{l,-}f\big]\e\big[\tilde{A}_{k,t}^{l+1,+}f\big].$
Furthermore,
\[
0=\E{\tilde{A}_{0,t}f}\E{\tilde{A}_{k,t}f}=\E{\tilde{A}_{0,t}^{l,-}f}\E{\tilde{A}_{k,t}^{l+1,+}f}+R_{k,t},
\]
where 
\[
|R_{k,t}|^{2}\le C_{T}e^{-\beta_{T}(t)\left[\left(\sqrt{k}-\sqrt{2}\right)\vee1\right]}
\]
for all $t\in[0,T]$ and $k\ge1$. Combining the estimates above,
we have shown that 
\begin{align*}
\left|\cov\left(\tilde{A}_{0,t}f\tilde{A}_{k,t}f\right)\right|&\le\left|\E{\tilde{A}_{0,t}f\tilde{A}_{k,t}f}-\E{\tilde{A}_{0,t}f}\E{\tilde{A}_{k,t}f}\right|\\
&\le C_{T}e^{-\frac{\beta_{T}(t)}{2}\left[\left(\sqrt{k}-\sqrt{2}\right)\vee1\right]}.
\end{align*}
Using the dominated convergence theorem and the fact that $t\beta_{T}(t)\to\frac{1}{8\sqrt{2}},$
we conclude that 
\[
\frac{1}{t}\sum_{k=1}^{\infty}\left|\cov\left(\tilde{A}_{0,t}f,\tilde{A}_{k,t}f\right)\right|\le \frac{1}{t}\sum_{k=1}^{\infty}C_{T}e^{-\frac{\beta_{T}(t)}{2}\left[\left(\sqrt{k}-\sqrt{2}\right)\vee1\right]}\to0
\]
as $t\to0+$. This completes the proof of the statement. 
\end{proof}

\subsection*{Acknowledgements}

Vitalii Konarovskyi was partly supported by the Deutsche Forschungsgemeinschaft (DFG, German Research Foundation) – SFB 1283/2 2021 – 317210226. He also thanks the Max Planck Institute for Mathematics in the Sciences for its warm hospitality, where a part of this research was carried out.

\setlength{\bibsep}{0pt}

\providecommand{\bysame}{\leavevmode\hbox to3em{\hrulefill}\thinspace}
\providecommand{\MR}{\relax\ifhmode\unskip\space\fi MR }
\providecommand{\MRhref}[2]{%
	\href{http://www.ams.org/mathscinet-getitem?mr=#1}{#2}
}
\providecommand{\href}[2]{#2}

\end{document}